\documentclass[11pt,twoside]{article}
\usepackage{latexsym,amssymb}
\usepackage{amsmath,amsthm}
\usepackage{a4wide}
\usepackage{fourier} 
\usepackage[svgnames]{pstricks}
\usepackage{graphicx}
\usepackage{titlesec}
\usepackage{fancyhdr,caption}
\usepackage{hyperref}
\usepackage{srcltx} 
\usepackage{subcaption}

\hypersetup{colorlinks=true,linkcolor=violet,citecolor=purple}
\titleformat*{\section}{\bfseries\scshape\Large}
\titleformat*{\subsection}{\bfseries\scshape\large}
\titleformat*{\subsubsection}{\itshape}

  \newrgbcolor{headcolour}{0.6 0.45 0.6}

\newcommand{\tr}{\mathop\mathrm{tr}\nolimits}
\newcommand{\Ad}{\mathop\mathrm{Ad}\nolimits}

\theoremstyle{plain}
\newtheorem{thm}{Theorem} 

\newtheorem{prop}[thm]{Proposition}
\newtheorem{cor}[thm]{Corollary}
\theoremstyle{definition}
\newtheorem{definition}[thm]{Definition}

\theoremstyle{remark}
\newtheorem{remark}[thm]{\bf{Remark}}

\newcommand{\SL}{\mathsf{SL}}

\renewenvironment{abstract}{%
 \setlength{\parindent}{0pt} \par\vskip\baselineskip\darkgray \hrulefill\par \smallskip \centerline{\large\scshape Abstract}\par\medskip}%
 {\par\medskip\hrulefill\bigskip}

\def\myparagraph{\paragraph} 

\begin{document}
\thispagestyle{empty}
{
\setlength{\parindent}{0pt}

{\Huge Point vortices on the hyperbolic plane}

\vskip 2\baselineskip
\color{DimGrey}
\textsc{\Large James Montaldi \& Citlalitl Nava-Gaxiola}\\[6pt]
\textsc{\large University of Manchester}

\vskip \baselineskip

{\large July 31, 2014}
}

\begin{abstract}
We investigate the dynamical system of point vortices on the hyperboloid. This system has non-compact symmetry $SL(2,  R)$ and a coadjoint equivariant momentum map. The relative equilibrium conditions are found and the trajectories of relative equilibria with non-zero momentum value are described. We also provide the classification of relative equilibria and the stability criteria for a number of cases, focusing on 2 and 3 vortices. Unlike the system on the sphere, this system has relative equilibria with non-compact momentum isotropy subgroup, and these are used to illustrate the different stability types of relative equilibria.
\end{abstract}


\section*{Introduction} 
Relative equilibria in systems of point vortices have previously been considered in detail on the plane and on the sphere. A thorough historical summary of research of these studies can be found in \cite{Aref-playground,crystals,{Kidambi}, NewtonBook} for the plane, and \cite{Laurent} for the sphere. 

On the other hand, the case of point vortices on the hyperbolic plane has only been treated briefly in \cite{{crystals}, {Boatto}, {Kimura},{Montaldi1},{Montaldi2}} and in some greater detail in \cite{Hwang, Hwang2}, although none take advantage of the geometry of the conserved quantities.  As on the plane and sphere, the governing equations of the system of point vortices on the hyperbolic plane are Hamiltonian.   Kimura \cite{Kimura} gives a uniform formulation for vortex motion on the sphere (positive curvature) and on the hyperbolic plane (negative curvature) and discusses the motion of vortex dipoles (pairs of vortices with opposite vorticity).   Deforming the phase space rather than the dynamics Boatto \cite{Boatto} and Montaldi and Tokieda \cite{Montaldi2} show how the curvature affects the stability conditions of a ring with $N$ vortices: for a given radius of ring, the ring becomes Lyapunov stable as the curvature decreases and for $N>7$ stability only occurs for negative curvature. Rings of vortices on the hyperbolic plane are also mentioned briefly in \cite{crystals}, and a more in-depth study is provided by Hwang and Kim \cite{Hwang2}. In that paper, they present conditions for  relative equilibria for rings of vortices on the hyperbolic plane, and also show that any two point vortex configuration is a relative equilibrium. 

The fixed and relative equilibria of three point vortices on the hyperbolic plane were first presented by Hwang and Kim in \cite{Hwang}. In the present paper we recover these relative equilibrium conditions using the symmetries of the system. The basic result is that relative equilibria fall into two broad classes: either the configurations form equilateral triangles or the three points lie on a geodesic (we call these geodesic relative equilibria). This is entirely analogous to the situation on the plane or the sphere.

Our principal aim is to study the stability of these relative equilibria, and one of the motivations for this is that the symmetry group $\SL(2,\mathbb{R})$ is not compact. The three conserved quantities form the components of the momentum map $\mathbf{J}$, and the symmetry properties of this map allow one to divide relative equilibria with non-zero momentum into three principal classes: elliptic, parabolic and hyperbolic, according to their momentum value, and this plays an important role in questions of stability. 

The paper is organized as follows.  In Sec.\,\ref{gem} we recall the basic geometry of the hyperbolic plane and its group of symmetries $\SL(2,\mathbb{R})$ and in particular we discuss the coadjoint action of this group, needed for the geometry of the momentum map. In Sec.\,\ref{sec:RE} we begin with a brief discussion of 2 vortices, in which every motion is a relative equilibrium, in order to illustrate the dynamical relevance of the three classes of momentum value. The main part of Sec.\,\ref{sec:RE} discusses the case of 3 point vortices, mostly recovering results from \cite{Hwang}, but adding information on the momentum type.

Finally, in Sec.~\ref{sec:stability} we discuss the different types of stability results for relative equilibria of two and three point vortices. We show that every two point vortex configuration is stable relative to $\SL(2,\mathbb R)$.  However, there is a finer notion of stability, namely stability relative to the subgroup $\SL(2,\mathbb R)_\mu$ (this is the isotropy subgroup for the momentum value $\mu$), and this only holds when the momentum value is elliptic, which in turn is true if the vortex strengths are of the same sign or, if they are of opposite signs, the vortices are not too far apart, see Theorem\,\ref{thm:stability of 2 point vortices}.  

For the stability of relative equilibria of three point vortices, we find remarkably that  an equilateral three vortex configuration has the exact same stability conditions of those for systems on the plane and on the sphere, namely that they are stable whenever $\sum_{i<j} \Gamma_i\Gamma_j>0$; here again stability is relative to the subgroup $\SL(2,\mathbb R)_\mu$. For geodesic relative equilibria the results are incomplete due to the complexity of the equations.  We prove in Theorem \ref{georelt} that the momentum value of any geodesic relative equilibrium is either zero or elliptic, and in Section \ref{stathree} provide some graphs showing the stability regions for isosceles configurations.  

This work is the major part of the PhD thesis \cite{thesis} of one of the authors (CN-G), and details omitted from this paper can be found in the thesis.

\section{Geometry \& equations of motion} \label{gem}

We begin by recalling the hyperboloid model we use for the hyperbolic plane. Alternative models, such as the Poincar\'e disc and the upper half plane are of course equivalent, but the hyperboloid model lends itself to a more straightforward representation of the momentum map. 

\myparagraph{Hyperboloid model}
The \emph{hyperboloid model} $\mathcal{H}_{2}$ of the hyperbolic plane is represented by the upper sheet of the 2-sheeted hyperboloid,
\begin{center}
$\mathcal{H}_{2}=\left\{\left(x,y,z\right) \in \mathbb{R}^3 \mid z^2-x^2-y^2=1, \, z>0\right\},$
\end{center}
with the Riemannian metric $ds^2_{\mathcal{H}_2}=dx^2+dy^2-dz^2$. This metric on $\mathbb R^3$ induces the \emph{hyperbolic inner product} $\langle \cdot, \cdot \rangle_{\mathcal H}$ between $X_1=\left(x_1, y_1, z_1\right)$ and $X_2=\left(x_2, y_2, z_2\right)$ in $\mathbb R^3$ given by 
\begin{equation}
\langle X_1, X_2\rangle_{\mathcal{H}} =x_1x_2+y_1y_2-z_1z_2,
\label{innerpro}
\end{equation}
and the \emph{hyperbolic cross product} 
\begin{equation*}
X_1\times_{\mathcal{H}}X_2=\left(y_1z_2 - z_1y_2,\,z_1x_2 - x_1z_2,\, -x_1y_2 + y_1x_2\right).
\label{cross}
\end{equation*}
Any \emph{geodesic} of this model is given by the curve of intersection of $\mathcal{H}_{2}$ with a plane through the origin \cite{{Cannon}, {Hwang}}. The \emph{hyperbolic distance} $d\left(X_1, X_2\right)$, between $X_1$ and $X_2 \in \mathcal H_2$, is naturally defined as the path length in the hyperbolic metric of the geodesic connecting these two points. A well known result  \cite{{Cannon}} relates the hyperbolic inner product to the hyperbolic distance by 
\begin{equation} \label{eq:cosh relation}
\langle X_1, X_2\rangle_{\mathcal H}= -\cosh \left(d \left(X_1, X_2\right)\right). 
\end{equation}

\myparagraph{Symmetry group of the hyperbolic plane}
The symmetry group of the hyperbolic plane is $\SL(2,\mathbb R)$, the group of real $2\times2$ matrices of unit determinant (in fact $\pm I$ both act trivially, so one usually says the symmetry group is the quotient $\mathsf{PSL}(2,\mathbb R)$, but we ignore this trivial point throughout).  Explicitly, the action in the hyperboloid model is given as matrix multiplication using the map
\begin{equation}\begin{array}{rcl}
\widetilde{  }\,: \SL(2,\mathbb R) &\rightarrow& \mathcal M\left(3, \mathbb R\right) \\
g=\left( 
\begin{array}{cc}
a & b \\ 
c & d%
\end{array}%
\right) &\mapsto & \widetilde g =\frac{1}{2} \left( 
\begin{array}{ccc}
2(ad+bc) & -2(ac-bd) & -2(ac+bd) \\ 
-2(ab-cd) & a^{2}-b^{2}-c^{2}+d^{2} & a^{2}+b^{2}-c^{2}-d^{2} \\ 
-2(ab+cd) & a^{2}-b^{2}+c^{2}-d^{2} & a^{2}+b^{2}+c^{2}+d^{2}
\end{array}%
\right),
\end{array}
\label{slmat}
\end{equation}
where $\mathcal M\left(3, \mathbb R\right)\subset GL(3,\mathbb R)$ is the group of  normalised M\"obius transformations \cite{Marden}.  That is, given $g\in \SL(2, \mathbb R)$ then $g\cdot X=\widetilde gX$. It is well-known that the action of $\SL(2, \mathbb R)$ on the hyperboloid $\mathcal H_2\subset \mathbb R^3$ is transitive and proper. 

The Lie algebra $\mathfrak {sl}\left(2, \mathbb R\right)$ of $\SL(2, \mathbb R)$ is given by the set of $2\times 2$ real matrices with zero trace, and we use the basis of $\mathfrak {sl}\left(2, \mathbb R\right)$ given by
\begin{equation}
\mathcal B=\left\{e_1=\left(
\begin{array}{c c}
1 & 0 \\
0 & -1
\end{array}\right), e_2=\left(
\begin{array}{c c}
0 & 1 \\
1 & 0
\end{array}\right), e_3=\left(
\begin{array}{c c}
0 & 1 \\
-1 & 0
\end{array}\right)\right\}. \label{basisliealg}
\end{equation}
Furthermore, one can identify the dual space $\mathfrak{sl}(2,\mathbb R)^*$ with the same set of trace zero $2\times 2$ matrices using the natural pairing 
\begin{equation}
\langle \mu,\, \xi\rangle=\frac{1}{2}\tr \left(\xi \mu\right).
\label{pairing}
\end{equation} 
The basis $\mathcal{B}'$ of $\mathfrak{sl}(2,\mathbb R)^*$ dual to $\mathcal B$ is  then given by the transposes of the elements of $\mathcal{B}$. Throughout, for elements of $\mathfrak{sl}(2,\mathbb R)$ and its dual we identify $\check{\rho}=\left(x, y, z\right)$ in $\mathbb R^3$ with a $2\times 2$ traceless matrix $\rho$ by 
\begin{equation}
\check{\rho}=\left(\begin{array}{c}
x \\ y\\ z \end{array}\right) \longleftrightarrow 
\rho=\left(
\begin{array}{c c}
x & y+z\\
y-z & -x 
\end{array}
\right).
\label{ident}
\end{equation}

This choice of basis and the vector space isomorphism (\ref{ident}) clearly associates $\xi \in \mathfrak {sl}\left(2, \mathbb R\right)$ and $\mu \in \mathfrak{sl}\left(2, \mathbb R\right)^*$ with $\check \xi$ and $\check \mu\in \mathbb R^3$, respectively. The matrix commutator then satisfies $\left[\xi, \eta \right]\check{ }=-2\left(\check \xi \times_{\mathcal H}\check \eta \right)$, and therefore (\ref{ident}) is a Lie algebra isomorphism, hence $ \mathfrak{sl}\left(2, \mathbb R\right)\cong \mathbb R^3$ with (-2 times) the hyperbolic cross product. Note also that $\det\rho = -\|\check\rho\|_{\mathcal H}^2$.

\myparagraph{Coadjoint action of $\SL(2,\mathbb R)$}
The non-degeneracy of the trace pairing (\ref{pairing}) implies that the adjoint and coadjoint actions of $\SL(2, \mathbb R)$ on $\mathfrak{sl}\left(2, \mathbb R\right)$ and $\mathfrak{sl}\left(2, \mathbb R\right)^*$ are equivalent. The coadjoint action of $g \in \SL(2, \mathbb R)$ on $\mu \in\mathfrak{sl}\left(2, \mathbb R\right)^*$ is given by matrix multiplication,
\begin{equation}
\Ad_{g^{-1}}^*\mu = g\mu g^{-1}.
\end{equation}

Using the basis $\mathcal B'$ and the representation (\ref{slmat}), this action becomes simply,
\begin{equation} \label{mobact}
\bigl(\Ad_{g^{-1}}^*\mu\bigr)\check{\ } = \widetilde g \check \mu,
\end{equation} 
and the coadjoint orbits are (the connected components of) the level sets of $\det\mu$ or equivalently of $\Vert\check\mu\Vert_{\mathcal H}$.

In the following theorem we show that every $\mu \neq 0$ in $\mathfrak {sl}\left(2, \mathbb R\right)^*$ has a coadjoint isotropy subgroup $SL\left(2,\mathbb R\right)_\mu$ which is a 1-parameter subgroup generated by a M\"obius transformation. We define the \emph{type} of $\mu$ as follows: $\mu$ is said to be \emph{elliptic, hyperbolic or parabolic} whenever $\SL(2, \mathbb R)_\mu$ is generated by an elliptic, hyperbolic or parabolic M\"obius transformation, respectively. 

\begin{figure}
        \centering
        \begin{subfigure}{.3\textwidth}
                \centering
                \includegraphics[width=.85\textwidth]{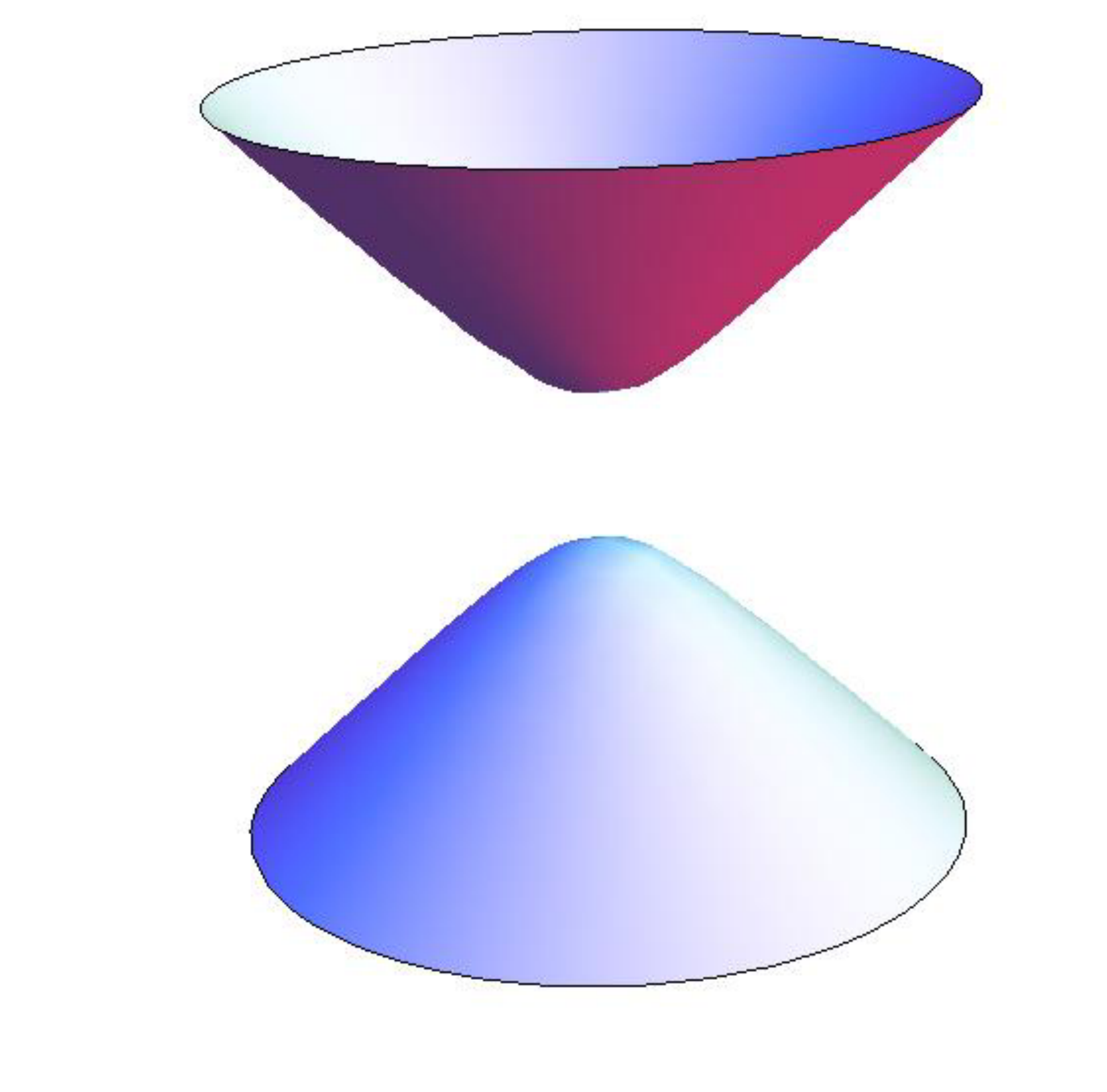}
                \caption{Each sheet of the two-sheeted hyperboloids is the coadjoint orbit for a $\check \mu$ inside $\mathfrak{C}$, that is $\det \mu>0$ and $\mu$ is {elliptic}.}
                \label{fig:posco}
        \end{subfigure}\hfill
        \begin{subfigure}{.3\textwidth}
                \centering
                \includegraphics[width=.85\textwidth]{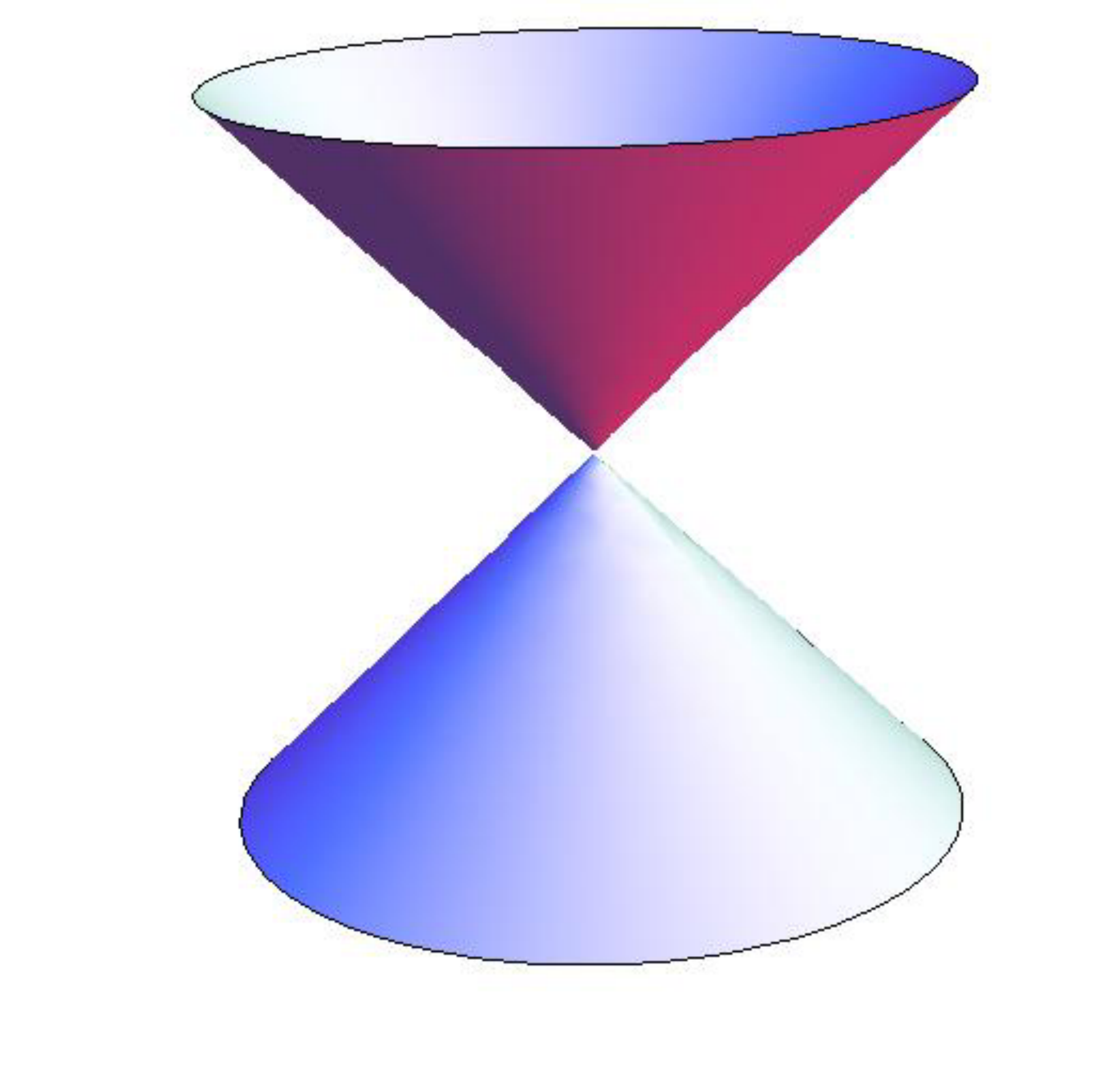}
                \caption{The null-cone $\mathfrak C$ without the origin is the union of the two coadjoint orbits for $\mu \neq 0$ such that $\det \mu=0$ and $\mu$ is parabolic. The origin itself is the coadjoint orbit of $\check \mu=0$. }
               \label{fig:zeroco}
        \end{subfigure}\hfill
        \begin{subfigure}{.3\textwidth}
                \centering
                \includegraphics[width=.85\textwidth]{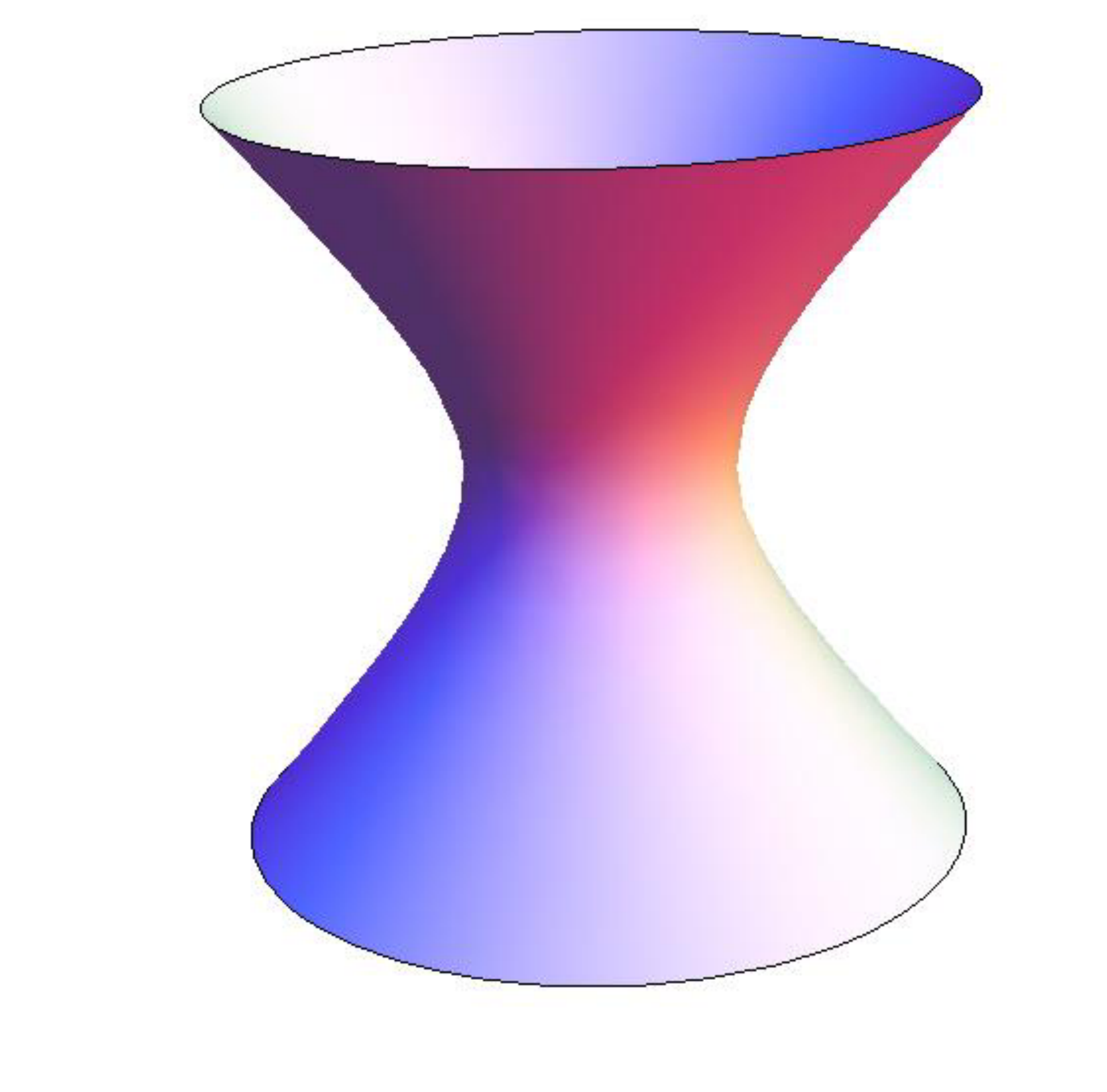}
                \caption{The coadjoint orbit for $\check \mu$ outside of $\mathfrak C$ is a hyperboloid of one sheet.  In this case $\det \mu<0$ and $\mu$ is hyperbolic.}
               \label{fig:negco}
        \end{subfigure}
        \caption{Coadjoint orbits of the action of $\SL(2, \mathbb R)$ in $\mathbb R^3$.}
       \label{coad}
\end{figure}

\begin{thm} 
\label{coadorb}
Let $\mathfrak C$ be the cone $\mathfrak C=\left\{\mu\in  \mathfrak{sl}\left(2, \mathbb R\right)^* \mid \det \mu=0\right\} \simeq \{X\in\mathbb R^3 \mid \|X\|_{\mathcal H}=0\}$. Then for $\mu\in  \mathfrak{sl}\left(2, \mathbb R\right)^*$ the coadjoint isotropy subgroups $\SL(2, \mathbb R)_\mu$ and coadjoint orbits are classified as follows:
\begin{enumerate}
	\item If\/ $\det \mu >0$ then  $\SL(2, \mathbb R)_\mu\cong SO(2,\mathbb{R})$, the type of $\mu$ is elliptic, and the coadjoint orbit is one sheet of the hyperboloid of two sheets shown in Figure \ref{fig:posco}.
		
	\item If\/ $\mu=0$ then $\SL(2, \mathbb R)_\mu=\SL(2,\mathbb{R})$ and the coadjoint orbit is the origin.
	
	\item If\/ $\det \mu=0$ and $\mu\neq0$ then $\SL(2, \mathbb R)_\mu\cong\left\{ \left( 
\begin{array}{cc}
1 & t \\ 
0 & 1%
\end{array}%
\right), t\in \mathbb{R}\right\}$. Here $\mu$ is parabolic and the coadjoint orbit is one sheet of\/ $\mathfrak{C}$ with the origin removed.
 \item If\/ $\det \mu<0$ then $\SL(2, \mathbb R)_\mu\cong \left\{ \left( 
\begin{array}{cc}
t & 0 \\ 
0 & t^{-1}%
\end{array}%
\right), t\in \mathbb R^{+} \right\}$,  $\mu$ is hyperbolic and the coadjoint orbit is a one sheeted hyperboloid as shown in Figure \ref{fig:negco}.
\end{enumerate}
Here $\cong$ means conjugate subgroups of\/ $\SL(2, \mathbb R)$. 
\end{thm}

\begin{proof}
The case $\mu=0$ is trivial. For $\mu \neq 0$, the proof consists in showing that given $X_1$ with the same sign of determinant then $G_\mu \cong G_{X_1}$. 

Consider $\mu$ with positive determinant, that is, the vector $\check \mu$ is inside the null-cone $\mathfrak C$. The null-cone $\mathfrak C$ is asymptotic to $\mathcal H_2$ and is the boundary of all vectors of this type. Therefore the line through $\check \mu$ intersects $\mathcal H_2$ at some point $\check \mu^\prime$, and there always exists $k\neq 0$ such that $\check \mu ^\prime =k \check \mu \in \mathcal H_2$, which consequently implies $G_\mu = G_{\mu^\prime}$. 

Let $X_1 =\left(0, 0, 1\right) \in \mathcal H_2$. It is not hard to show that there exists $g\in SL\left(2,\mathbb R\right)$ such that $g\cdot X_1=\Ad_g X_1=\mu ^\prime$. This implies $g G_{X_1} g^{-1}=G_{\mu ^\prime}$ and $G_{X_1}\cong G_{\mu ^\prime}= G_\mu$. The result is now easily obtained by calculating the isotropy subgroup of $\check {X_1}=\left(0, 0, 1\right) \in \mathcal H_2$.

A similar argument works in the other two cases, taking $\check X_1=(0,1,1)$ for the parabolic case and $\check X_1=(1,0,0)$ for the hyperbolic case.
\end{proof}

There is an appealing geometric description of the $\SL(2,\mathbb R)_\mu$-orbits in $\mathcal H_2$ as follows.  Let $P_\mu$ denote the hyperbolic normal plane passing through $\check \mu$,
$$P_{\mu}:=\left\{\check X\in\mathbb{R}^{3} \mid 
\left\langle \check X - \check \mu,\;\check \mu \right\rangle_{\mathcal{H}} =0\right\}.$$  
It follows from $\left(\ref{mobact}\right)$, and the fact that the inner product $\left(\ref{innerpro}\right)$ is invariant under the coadjoint action that $P_{\mu}$ is invariant under the action of $\widetilde g$ for $g\in \SL(2, \mathbb R)_\mu$.

Since $\mathcal H_2$ is itself invariant under the coadjoint action, a simple  consequence of this is that the curve $P_{\mu} \cap \mathcal H_2$ remains invariant under the coadjoint action of $\SL(2, \mathbb R)_\mu$. It follows that if $\nu\in P_{\mu} \cap \mathcal H_2$ then the $\SL(2,\mathbb R)_\mu$-orbit of $\nu$ is this curve of intersection.  Furthermore, this curve is a conic of the type (ellipse, hyperbola or parabola) related to its isotropy group, and coincides with the type of $\mu$ as defined in Theorem\,\ref{coadorb}. This can be seen in Figure \ref{conics}.  For other points $\nu\in\mathcal H_2$, one can replace $\mu$ by a scalar multiple of $\mu$ (which necessarily has the same isotropy group), which replaces $P_\mu$ by a parallel plane, whose intersection with $\mathcal H_2$ is the $\SL(2,\mathbb R)_\mu$-orbit of $\nu$.

\begin{figure}
        \centering
        \begin{subfigure}[b]{0.30\textwidth}
                \centering
                \includegraphics[width=\textwidth]{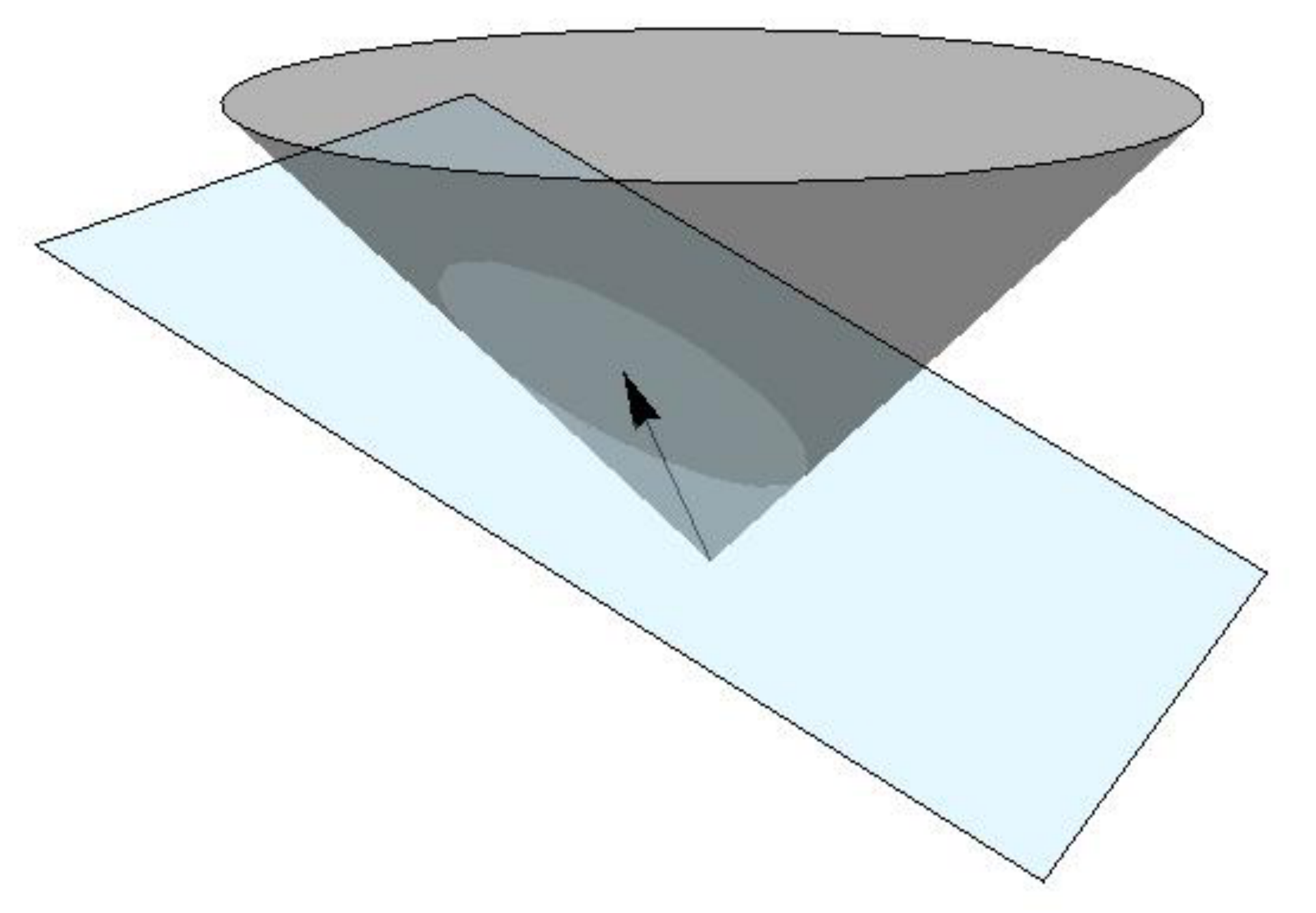}
                \caption{$\check \mu$ inside $\mathfrak{C}$ (elliptic)}
                \label{fig:ellip}
       
        \end{subfigure}
        ~ 
                \begin{subfigure}[b]{0.32\textwidth}
                \centering
                \includegraphics[width=\textwidth]{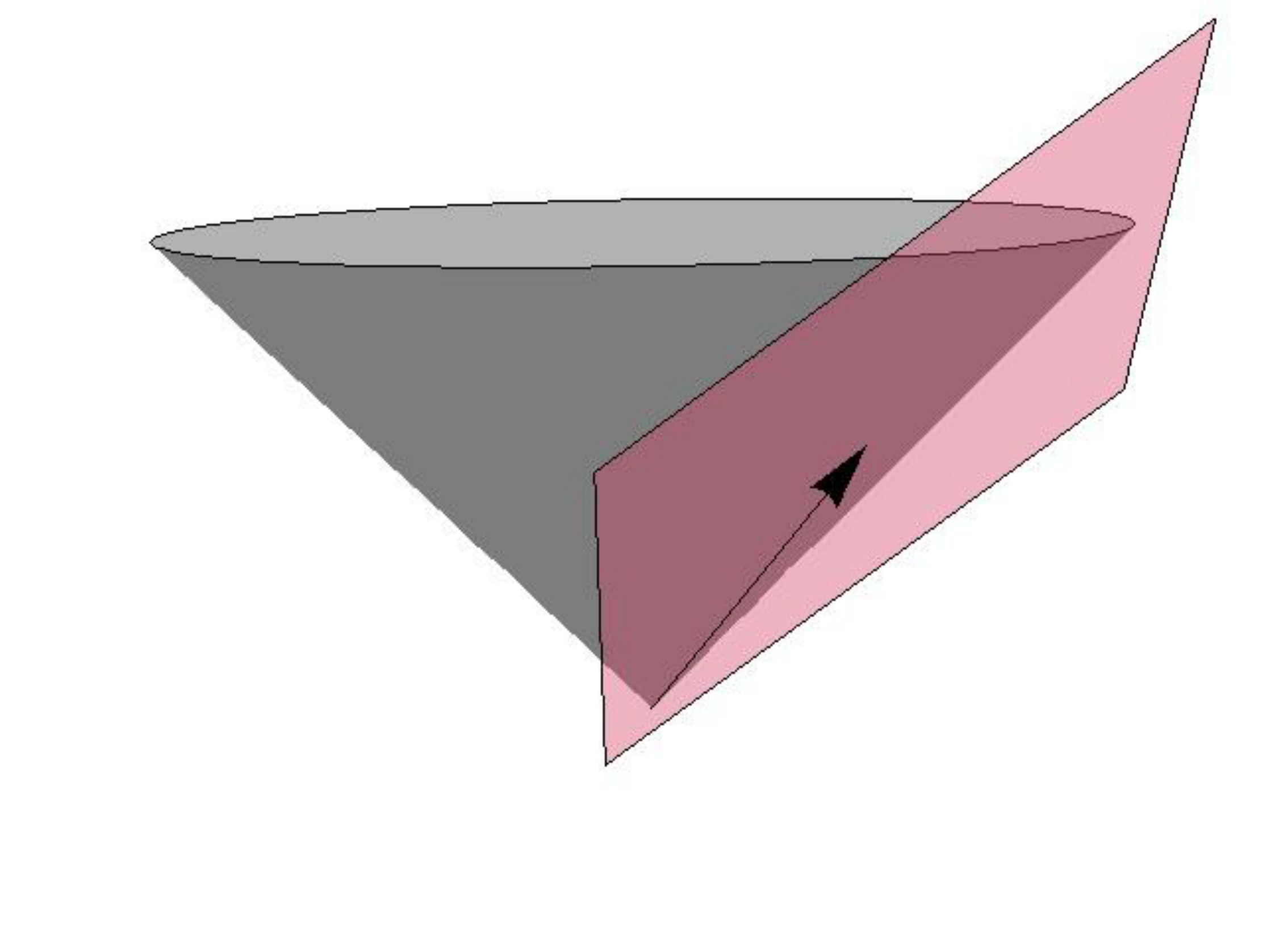}
                \caption{$\check{\mu}$ on $\mathfrak{C}$ (parabolic)}
                \label{fig:parab}
        \end{subfigure}
        ~ 
        \begin{subfigure}[b]{0.34\textwidth}
                \centering
                \includegraphics[width=\textwidth]{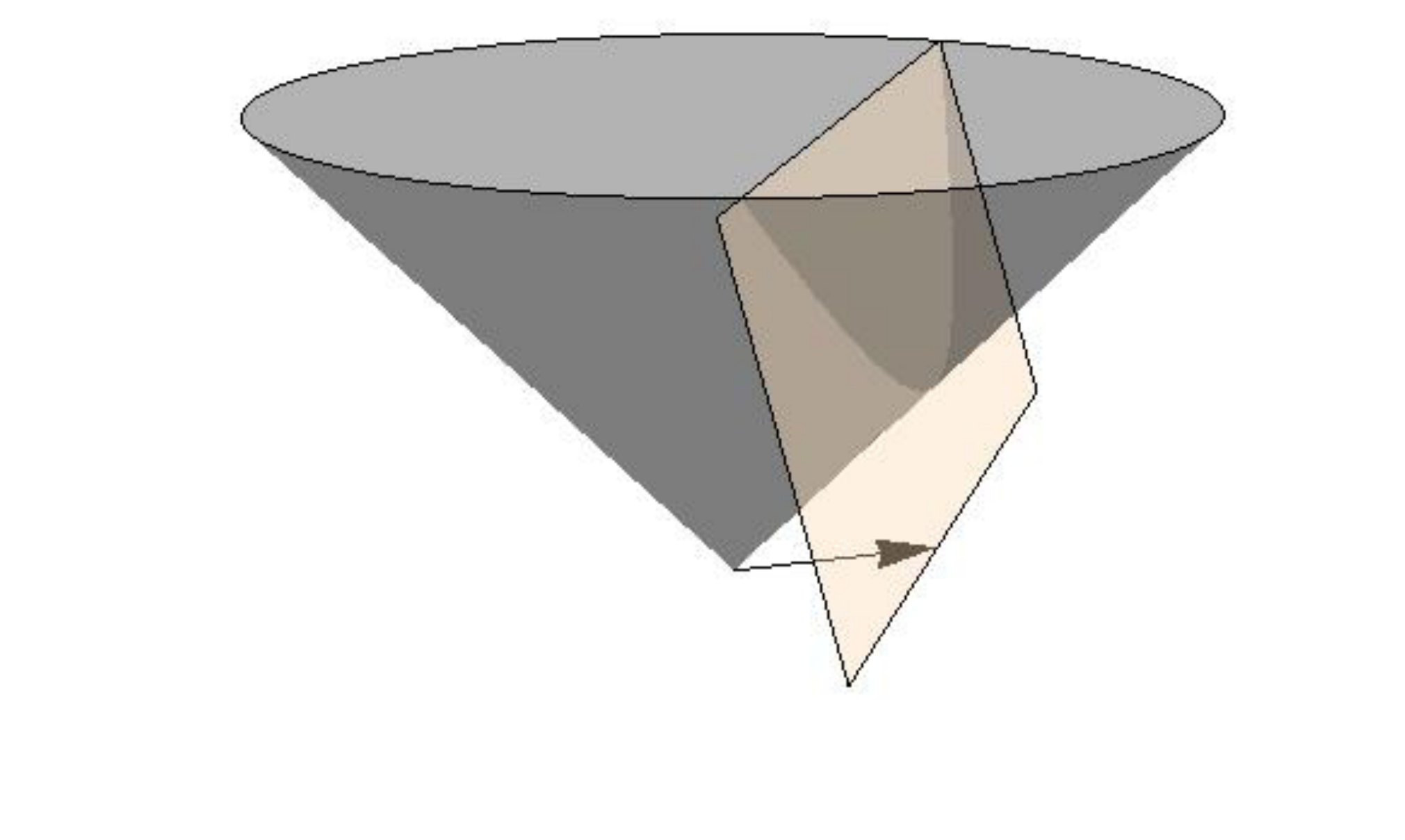}
                \caption{$\check{\mu}$ outside $\mathfrak{C}$ (hyperbolic)}
               \label{fig:hyperb}
          
        \end{subfigure}
  \caption{Intersection of the hyperbolic normal plane $P_{\mu}$ with $\mathcal{H}_2$, classified in terms of the position of $\check \mu$ with respect to the null-cone $\mathfrak C=\left\{\check \mu \mid \det \mu=0\right\}$.}
  \label{conics}
\end{figure}

\myparagraph{Phase space} We now return to the system of $N$ point vortices on the hyperbolic plane $\mathcal H_2$. 
Let $\check X_{i}$ be the vector from the origin in $\mathbb R^3$ to the $i^{\mathrm{th}}$ vortex $X_i\in\mathcal H_2$, and denote its vorticity by $\Gamma_{i}$ (assumed to be nonzero). A candidate for the manifold of the dynamical system of $N$ point vortices consists of $N$ copies of the hyperboloid $\mathcal H_2 \times \dots  \times \mathcal H_2$. However, configurations that lead to infinite energy must be avoided, and this is achieved by discarding the set of collisions \[\Delta=\left\{X=\left(X_1, \dots , X_N\right) \in \mathcal H_2 \times \dots \times \mathcal H_2 \mid \text{two or more $X_i$ coincide}\right\}.\] Hence, the phase space is given by $\mathcal M=\mathcal H_2 \times \dots \times \mathcal H_2 \setminus \Delta$.
Neglecting collisions of vortices guarantees the action of $\SL(2, \mathbb R)$ on $\mathcal M$ to be free. 

The dynamics on $\mathcal M$ are determined by the vector field
\begin{equation}\label{velocity}
\dot X_r = \frac{1}{\pi} \sum_{p\neq r}\Gamma_p \frac{\check X_p \times_{\mathcal H} \check X_r}{\left\langle \check X_{r}, \check X_{p} \right\rangle_{\mathcal H}^2-1}
\end{equation}
with $r \in \left\{1,.., N\right\}$.
This equation differs from the differential equations derived by Kimura in \cite{Kimura} by a factor of 2, which is explained by the choice of the basis $\mathcal B$ of the Lie algebra in (\ref{basisliealg}). 

This system is Hamiltonian, where the symplectic structure on $\mathcal H_2$ is derived from the natural Lie-poisson structure on $\mathfrak{sl}(2,\mathbb R)^*$ given by 
$$\{F, G\}(\mu) = \left\langle \mu, \, \left[dF(\mu), \,dG(\mu) \right] \right\rangle,$$ 
where $F,G$ are two smooth functions on $\mathfrak{sl}(2,\mathbb R)^*$, and $dF(\mu)\in(\mathfrak{sl}(2,\mathbb R)^*)^* \simeq \mathfrak{sl}(2,\mathbb R)$. 
The resulting symplectic structure on the coadjoint orbit $\mathcal H_2$, called the Kostant-Kirillov-Souriau (KKS) form is given as follows.  Let $\mu \in \mathcal H_2$ and $u,v\in T_\mu \mathcal H_2$  then the KKS form is 
\begin{equation}
\omega_{\mathcal H_2} \left(\mu\right)\left(u, v\right)=  \frac{\check \mu \cdot \left(\check u\times_{\mathcal H} \check v\right)}{2\Vert \check \mu \Vert^2},
\label{kkshyp}
\end{equation}
where $\Vert \check \mu \Vert^2=\sum_{i=1}^3 \mu_i^2$ is the Euclidean norm. The symplectic from on $\mathcal M$ depends on the vorticities as
 \begin{equation} \omega_\mathcal M\left(\cdot, \cdot\right)=\sum_{i=1}^{N} \Gamma_i \omega_{\mathcal H_2}\left(\cdot, \cdot\right)_i.\label{sympn}
\end{equation}

The Hamiltonian for the system is as constructed by Kimura \cite{Kimura}, which in terms of the hyperbolic inner product is given by 
\begin{equation}
H=-\frac{1}{4\pi}\sum{\Gamma_i \Gamma_j \ln\frac{\langle \check X_i, \check X_j\rangle_{\mathcal H}+1}{\langle \check X_i, \check X_j\rangle_{\mathcal H}-1}}.
\label{hamil}
\end{equation}
Note that if all vorticities have the same sign, as two point vortices in $\mathcal H_2$ get closer, i.e.\ as the hyperbolic distance between them tends to $0$, the Hamiltonian tends to $\infty$, as expected when a collision occurs. On the other hand if two points get far apart, the hyperbolic distance tends to $\infty$ and the contribution of their interaction to the total energy is $0$.

\myparagraph{Momentum map and its equivariance}
This is central to our analysis, and this is where the hyperboloid model is so useful. The momentum map is $\mathbf{J}:\mathcal{M} \to \mathfrak {sl}\left(2, \mathbb R\right)^*$ is given simply by
\begin{equation}
 \mathbf{J}\left( X_1, \dots , X_N\right)= \sum_{i=1}^N \Gamma_i  X_i.
\label{momentum}
\end{equation}
See for example \cite{Montaldi2} for details, and how the same formula can be made to hold for vortices in the Euclidean plane. 

Whenever the symmetry group is semisimple the momentum map of a symplectic manifold can be chosen to be coadjoint equivariant, as was shown by Souriau \cite{Sou}.  Since the KKS form is invariant and it defines the symplectic structure $\left( \ref{sympn}\right)$, the momentum map $\left(\ref{momentum}\right)$ does satisfy this equivariance, that is
\[\mathbf{J}\left(g\cdot X\right)=\Ad^*_{g^{-1}} \mathbf{J}\left(X\right)\]
for all $g\in \SL(2, \mathbb R)$, $X\in \mathcal M$. By Noether's theorem the momentum map is a conserved quantity under the flow of every invariant Hamiltonian.

\section{Relative equilibria} \label{sec:RE}
Throughout this section we write $G=\SL(2,\mathbb{R})$ for convenience. 
A point $X_e$ in phase space is a \emph{relative equilibrium} if its group orbit is invariant under the dynamics, or equivalently, if the trajectory through $X_e$ is contained in the group orbit.  Since the momentum value $\mu=\mathbf{J}\left(X_e\right)$ is a conserved quantity (by Noether's theorem), the level sets $\mathbf{J}^{-1}\left(\mu\right)$ are invariant under the flow of the Hamiltonian system. This also means that we must restrict our attention to the action of $G_{\mu}$, implying that a configuration $X_e$ is a relative equilibrium if $G_\mu \cdot X_e$ remains invariant.

Consequently, if $\dim G_\mu=1$ then the trajectory of a relative equilibrium follows the $G_\mu$-orbit, and in our point vortex model this occurs whenever $\mu\neq0$.

\subsection{Two point vortices} \label{sec:2 point vortices}

For a two point vortex configuration $X_e=\left(X_1, X_2\right)\in \mathcal M$, the implicit function theorem implies $\dim \mathbf{J}^{-1}\left(\mu\right)=1$. We additionally know that $G_\mu \cdot X_e\subset \mathbf{J}^{-1}\left(\mu\right)$, so they must be (locally) equal, hence $G_\mu \cdot X_e$ is indeed invariant. In conclusion, any two point vortex configuration $X_e$ is a relative equilibrium. 

It also follows that any two vortex configuration has non-zero momentum value $\mu=\mathbf{J}\left(X_e\right)$ (for otherwise $G\cdot X_e$ is 3-dimensional and contained in $\mathbf{J}^{-1}(0)$). Thus, as shown in Figure \ref{trajectories}, the trajectories of $X_1$ and $X_2$ are the conics determined by the determinant of the momentum value $\mu$. 

\begin{figure}
        \centering
        \begin{subfigure}[b]{0.30\textwidth}
                \centering
                \includegraphics[width=\textwidth]{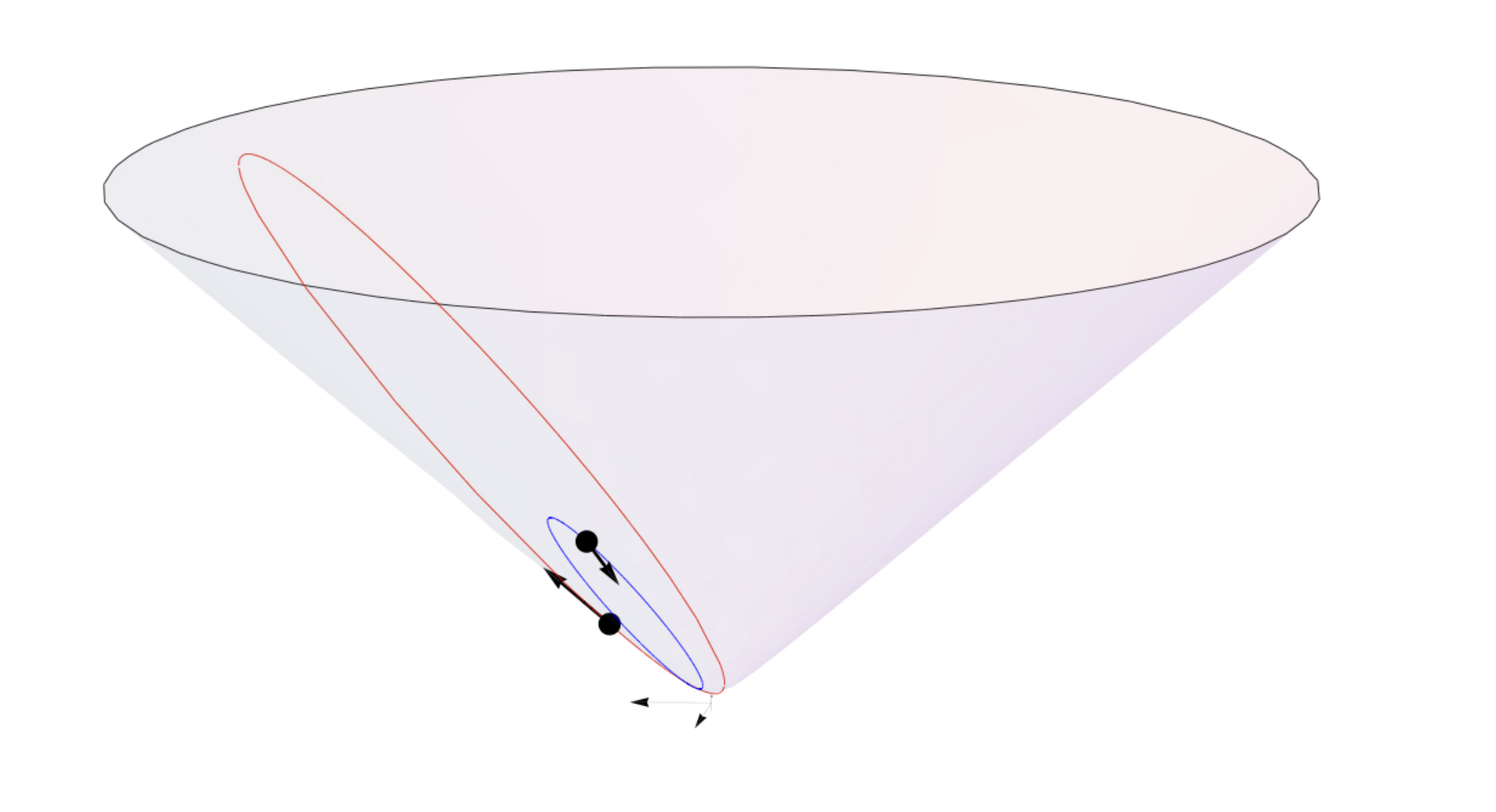}
                \caption{$\check{\mu}$ inside $\mathfrak{C}$, with $\Gamma_1=1$ and $\Gamma_2=3$.}
                \label{trajectory1}
        \end{subfigure}
        \hfill
        \begin{subfigure}[b]{0.30\textwidth}
                \centering
                \includegraphics[width=\textwidth]{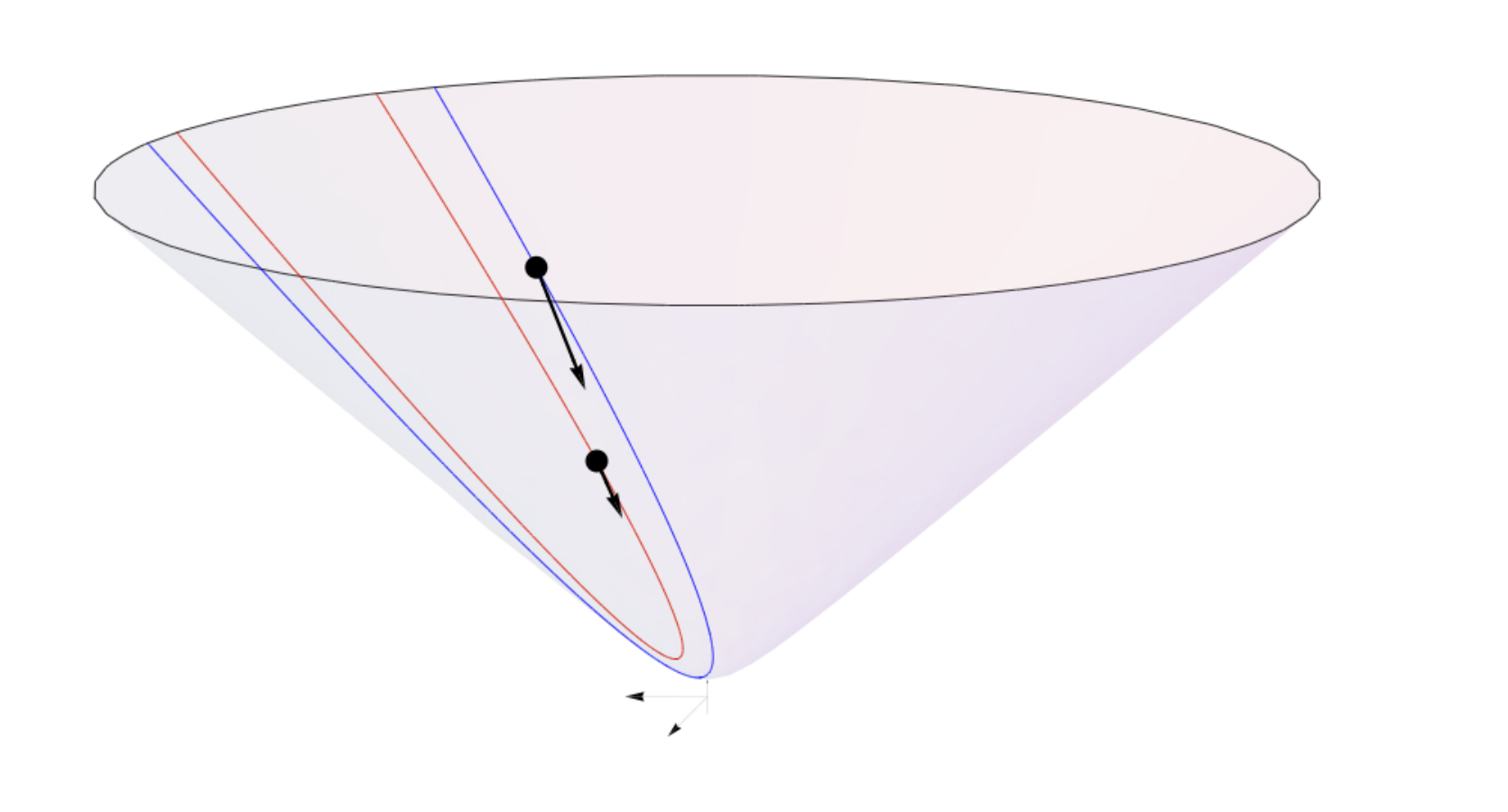}
                \caption{$\check{\mu}$ on $\mathfrak{C}$, with $\Gamma_1=1$ and \\
                $\Gamma_2=-\frac{1}{2}$.}
                \label{trajectory2}
        \end{subfigure}
        \hfill 
        \begin{subfigure}[b]{0.30\textwidth}
                \centering
                \includegraphics[width=\textwidth]{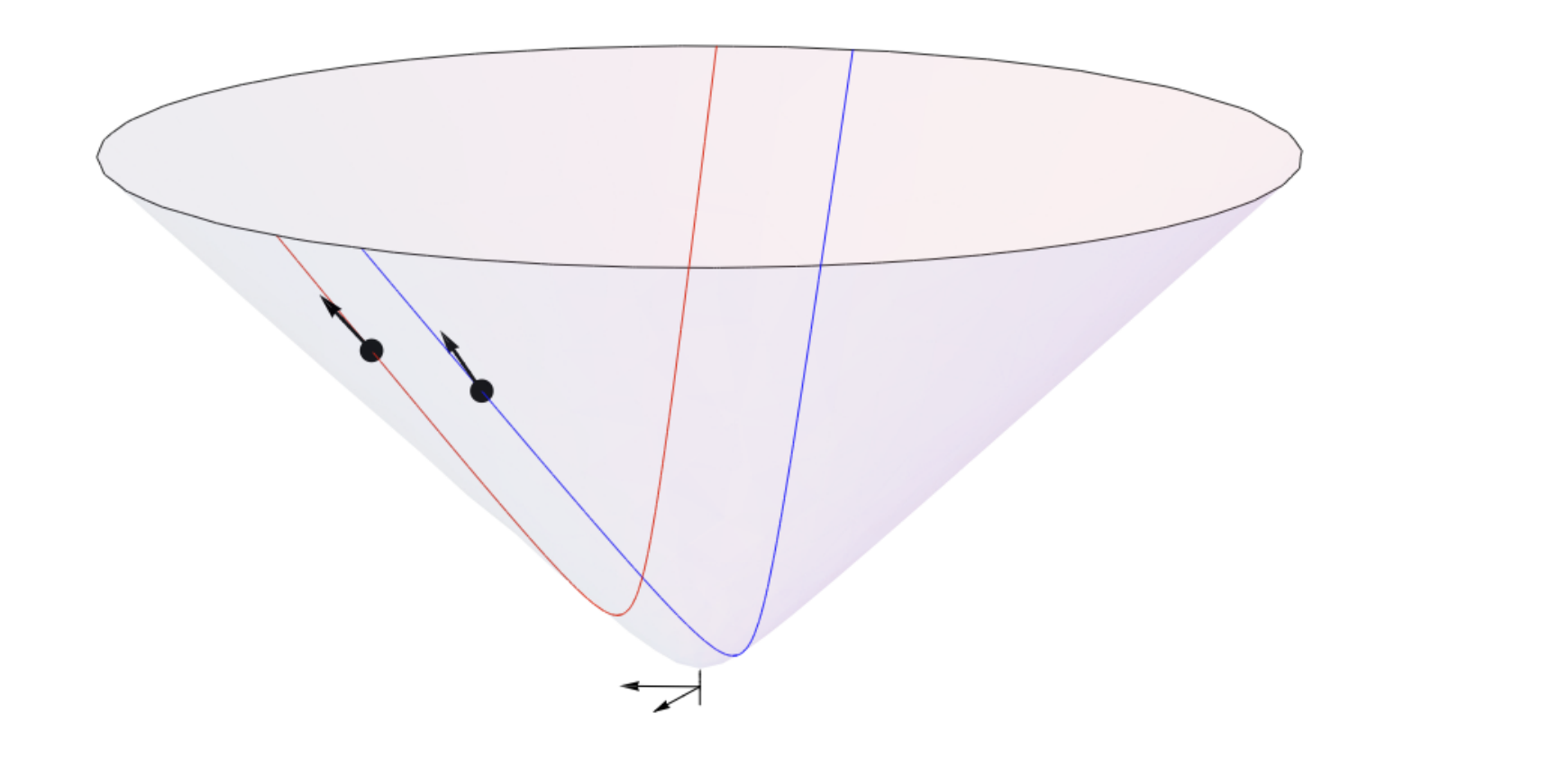}
                \caption{$\check{\mu}$ outside $\mathfrak{C}$, with $\Gamma_1=1$ and $\Gamma_2=-1$.}
               \label{trajectory3}
            \end{subfigure}
               \caption{Trajectories of the vortices for $\check{\mu}$ inside, outside and on the null-cone $\mathfrak{C}$.}
  \label{trajectories}
           \end{figure} 

For example the vortex dipole $\Gamma_1=-\Gamma_2=1$, treated before by Kimura \cite{Kimura} and recently by Hwang and Kim \cite{Hwang2}, $X_e$ has momentum value $\mu$ with determinant less than zero. Therefore $G_\mu$ is related to an hyperbolic M\"obius transformation, and the trajectories of the vortices are on hyperbolas parallel to each other (Figure \ref{trajectory3}) (and as Kimura points out, their midpoint follows a geodesic).

\subsection{Three point vortices}\label{relthreesec}
With more than two vortices the conditions for a relative equilibrium are not so straightforward. Considering the orbit space $\mathcal M/ \SL(2, \mathbb R)$, an invariant group orbit is just a point that is invariant under the dynamics. Thus a relative equilibrium $X_e$ is an equilibrium point in the reduced space. Given that the level sets $\mathbf{J}^{-1}\left(\mu\right)$ are invariant under the flow of Hamiltonian, a relative equilibrium must be a critical point of the restriction $H\vert_{\mathbf{J}^{-1}\left(\mu\right)}$, see for example \cite{Marsden-LoM}. Hence if there exists $\xi \in \mathfrak{sl}\left(2, \mathbb R\right)$ such that $X_e$ is a critical point of the \emph{augmented Hamiltonian} $H_{\xi}(X)=H(X)-\left\langle \mu,\xi\right\rangle$, then $X_e$  is a relative equilibrium.

In practice, we write $H$ as a function on $(\mathbb{R}^3)^N$, so must add the constraint that $X\in \mathcal M$.  We do this by including Lagrange multipliers $\lambda_r$, and thus $X_e=\left(X_1, \dots ,X_N\right)$ is a relative equilibrium if it is a critical point of
\begin{equation}
{H}_{\xi}\left(X_e\right)=
-\frac{1}{4\pi} \sum \limits_{r \neq s}^N\Gamma_{r}\Gamma_{s}\ln\left(\frac{\left\langle X_{r}, {X_{s}}\right\rangle_{\mathcal H}+1}{\left\langle {X_{r}}, {X_{s}}\right\rangle_{\mathcal{H}}-1}\right)- \sum \limits_{i=1}^3\sum \limits_{r=1}^N \tau_i \xi_i \Gamma_r X_r^i+\sum \limits_{r=1}^N \lambda_r \left(\langle {X_r}, {X_r}\rangle_{\mathcal H} +1\right)
\label{extended}
\end{equation}
where ${X_j}=\left(X_j^{1},X_j^{2},X_j^{3}\right)\in \mathcal H_2$, with vortex strengths $\Gamma_j$ (for $j=r,s$) and

\[\tau_i = \left\{ 
\begin{array}{l l}
  1 & \quad \mbox{if $i=1,2$, }\\
  -1 & \quad \mbox{if $i=3$,}\\ \end{array} \right. \]
and $\lambda_r$ is the corresponding Lagrange multiplier. 

Note that $\langle  X_r,  X_s\rangle_{\mathcal H} =\sum_{i=1}^3 \tau_i X^{i}_{r}X^{i}_{s}$, thus $\frac {\partial}{\partial X_r^i} \langle {X_r},{X_s}\rangle_{\mathcal H}= \tau_iX_s^i$, and
\begin{equation}
\frac{\partial {H}_{\xi}}{\partial X_r^i}=\tau_i\left(\frac{\Gamma_r}{
2\pi}\underset{p\neq r}{\sum}\Gamma_p \frac{X_p^i}{\left\langle  X_{r},  X_{p} \right\rangle_{\mathcal H}^2-1} -\Gamma_r \xi_i +\lambda_r X_r^i \right)=0.
\label{firstderiv}
\end{equation}

Therefore the general condition for relative equilibria is given as the solutions of the following equation of \emph{angular velocity $\xi$}
\begin{equation}
\xi_i=\frac{1}{2\pi}\underset{p\neq r} \sum{\Gamma_p \frac{X_p^i}{L_{pr}}}+\frac{\lambda_r}{\Gamma_r}X_r^i,    \qquad \text{$\forall r \in \left\{1,2, \dots N\right\}$ and $i \in\left\{1,2,3\right\},$}
\label{equireleq}
\end{equation}
where $L_{pr}$ denotes $\left\langle X_{p},X_{r}\right\rangle_{\mathcal H}^2-1=\sinh^2(d(X_r,X_p))$. An important observation is that the angular velocity $\xi$ satisfies $\dot{ {X}}_r= \xi \times_{\mathcal H}  {X}_r$. This means that a relative equilibrium $X_e$ rotates "hyperbolically" around $\xi$ as shown in \cite[Proposition 2]{Hwang}.

\myparagraph{Classification} 
In this section we provide the classification of relative equilibria of three point vortices, which is strikingly similar to the classification for the plane and the sphere. The following result is obtained by solving $\left(\ref{equireleq}\right)$ for 
\begin{eqnarray}
 {X_1}&=&\left(x_1,y_1,z_1\right), \nonumber \\
{X_2}&=&\left(0,0,1\right),  \label{threl}
\\
{X_3}&=&\left(x_3,y_3,z_3\right). \nonumber
\end{eqnarray}
Any other set of three point vortices is equivalent to this one by hyperbolic rotations. Since the dynamics are preserved by that type of transformation, the same relative equilibrium conditions follow for any other $X$.

\begin{thm}
\label{threerelt}
Every relative equilibrium of three point vortices in the hyperbolic plane is either an equilateral triangle or a geodesic configuration.
\end{thm} 

The result is obtained by calculating the solutions of (\ref{equireleq}) for configurations of the form given in (\ref{threl}).  

\begin{remark}
As mentioned in the introduction, a geodesic on the hyperboloid model $\mathcal H_2$ is the intersecting curve of a plane through the origin with $\mathcal H_2$. As pointed out in \cite{Hwang}, and unlike the system on the sphere, it is not possible to have an equilateral configuration in a geodesic of the hyperboloid model. Therefore, for an equilateral configuration 
$$V= \left\langle X_1,\, X_2 \times_{\mathcal H} X_3\right\rangle_{\mathcal H}\neq 0.$$
 Although not mentioned explicitly in \cite{Hwang}, their formulae (17--19) lead to the same result of relative equilibria derived here. 
Furthermore, $ X_1$, $ X_2$ and $ X_3$ are linearly independent in $\mathbb R^3$, implying that $\mathbf{J}\left(X_e\right)=\Gamma_1X_1+\Gamma_2 X_2+\Gamma_3 X_3 \neq 0$ for every equilateral configuration. Thus the $G_\mu$ orbit of an equilateral relative equilibrium is one of the conics described in Section~\ref{gem}. 
 \end{remark}

The next two theorems present the conditions on the vorticities $\Gamma$ for relative equilibria of three point vortices. The complete proofs in \cite{thesis} consist in finding any restrictions on $\Gamma$ for the existence of the angular velocity $\xi$ in $\left(\ref{equireleq}\right)$. 

\begin{thm} 
Every equilateral configuration $X_e$ of point vortices $\left(X_1, X_2, X_3\right)$ in $\mathcal M$ is a relative equilibrium. The angular velocity of $X_e$ is given by \begin{equation}\xi=\frac{1}{2\pi L}\mathbf{J}\left(X_1,X_2,X_3\right),
\label{angeq}
\end{equation}
where $L=\langle {X_i}, {X_j}\rangle_{\mathcal H}^2-1$ for $i\neq j$. 
\end{thm}

Note that $L=\sinh^2(d(X_i,X_j))$ as follows from (\ref{eq:cosh relation}).   Since for equilateral configurations $\mathbf{J} \neq0$, it follows from  (\ref{angeq}) that these are never equilibria (it is also easy to see this geometrically from first principles).

\begin{thm} 
Let $X_e=\left(X_1, X_2, X_3\right)\in \mathcal M$ be a configuration of point vortices on a geodesic of the hyperboloid, with vorticity vector $\Gamma=\left(\Gamma_1, \Gamma_2, \Gamma_3\right)$. Suppose $ X_2$ lies between $ X_1$ and $ X_3$ on the geodesic, and let $L_{ij}=\langle  X_i,  X_j\rangle_{\mathcal H}^2-1$. Then $X_e$ is a relative equilibrium point if and only if
\begin{equation} 
\sqrt {L_{{23}
}} \left( L_{{13}}-L_{{12}} \right) \Gamma _{{1}}+\sqrt{L_{13}} \left( L_{{23}}-L_{{12}} \right) \Gamma _{{2}}+\sqrt {L_{{12}}}
 \left( L_{{23}}-L_{{13}} \right) \Gamma _{{3}}=0
 \label{georel}
\end{equation}
Moreover, the momentum value $\mu$ of a geodesic relative equilibrium is either zero or elliptic.
\label{georelt}
\end{thm}

\begin{proof}
The proof of expression $\left(\ref{georel}\right)$ is derived by computing the relative equilibrium conditions $\left(\ref{equireleq}\right)$ for 
\begin{eqnarray}
{X_1}&=&\left(x_1,0,\sqrt{1+x_1^2}\right), \nonumber \\
{X_2}&=&\left(0,0,1\right), \nonumber \\
{X_3}&=&\left(-x_3,0,\sqrt{1+x_3^2}\right),
\label{geocon}
\end{eqnarray}
in terms of $L_{ij}$ where $x_1=\sqrt{L_{12}}$, $x_3=\sqrt{L_{23}}$.

For the type of momentum, we first consider isosceles geodesic configurations, that is $x_1=x_3$. Straightforward calculations show that for Equation (\ref{georel}) to be satisfied $\Gamma_1=\Gamma_3$ must hold. Conversely, substituting $L_{13}$ in terms of $L_{12}$ and $L_{23}$ in $\left(\ref{georel}\right)$ with $\Gamma_1=\Gamma_3$ leads to $L_{12}=L_{23}$. Under this vorticity condition, the determinant of the momentum value $\mu=\mathbf{J}\left(X_1, X_2, X_3\right)$ is
\begin{equation}
\det \mu=\left(2\Gamma_1 \sqrt{1+x_1^2}+\Gamma_2\right)^2,
\end{equation}
therefore $\det \mu>0$ for all $\Gamma_2\neq -2\Gamma_1\sqrt{1+x_1^2}$, otherwise $\mu=0$. Suppose now that $x_1\neq x_3$ and $\Gamma_1\neq \Gamma_3$, the determinant of the momentum value is
\begin{eqnarray*}
\det \mu&=&8 \frac{\left(\Gamma _{1}x_{1}-\Gamma _{3}x_{3}\right)^{2}}{k^2} \Biggl(\left(  \left( \frac{1}{4}+x_3^2 \right)x_1^2+\frac{1}{4}x_3^2\right) x_3x_1\sqrt{1+x_3^2}\sqrt {1+x_1^2}+ \Biggr. \\
&&\Biggl. + \left( \frac{1}{8}+\frac{3}{4}x_3^2+x_3^4\right)x_1^4+\left( \frac{3}{4}x_3^4+\frac{3}{8}x_
3^2\right)x_1^2+\frac{1}{8}x_3^4 \Biggr),
\end{eqnarray*}
where $k=\left(x_1-x_3\right)\left(x_1+x_3\right)\left(x_3\sqrt{1+x_1^2}+x_1\sqrt{1+x_3^2}\right)$. Recall that $x_1>0$ and $x_3>0$, hence $\mu$ is elliptic provided $x_1\Gamma _{1} \neq \Gamma _{3}x_3$, otherwise $\mu=0$.
\end{proof}

It is remarkable that for any geodesic configuration of three vortices $X_e$ there can always be found a set of vorticities $\Gamma$ such that $X_e$ is a relative equilibrium. Conversely, for given values of the vorticities there is a 1-parameter family of inequivalent geodesic relative equilibria.

\myparagraph{Equilibria} It is interesting to ask which of the relative equilibria are in fact (fixed) equilibria. This was answered by Hwang and Kim \cite{Hwang}, who showed that a necessary and sufficient condition for an equilibrium is that
 
\begin{equation} \label{eq:HK-equilibrium}
\sum_i\Gamma_i(\Gamma_j+\Gamma_k) X_i=0
\end{equation}
where $i,j,k$ is a cyclic permutation of $1,2,3$, and this is only possible if $\sum_{i<j}\Gamma_i\Gamma_j<0$.

A particular case is the isosceles geodesic equilibrium whose stability we will consider again at the end of the paper.  A calculation using (\ref{georel}) shows that if $L_{12}=L_{23}$ then $\Gamma_3=\Gamma_1$ for a relative equilibrium, and, as shown in the proof of Theorem \ref{georelt} above, its momentum value $\mu$ is elliptic if $a=\langle  X_1,  X_2 \rangle_{\mathcal H} \neq \frac{\Gamma_2}{2\Gamma_1}$, otherwise $\mu=0$. Taking the hyperbolic inner product of the condition (\ref{eq:HK-equilibrium})  with $ X_2$ and combining with the value of $a$ just given leads to 
\begin{equation} 
\Gamma_2=\frac{\Gamma_1a}{1-a}.
\label{relfix}
\end{equation}
Therefore, a configuration $X_e$ with $\Gamma_1=\Gamma_3$ and $\Gamma_2$ given by (\ref{relfix}) is an isosceles equilibrium configuration. Furthermore, the momentum value of $X_e$ is elliptic.

\section{Stability of relative equilibria}
\label{sec:stability}

Before presenting any of our stability results, we begin by recalling the notions of (nonlinear) stability for relative equilibria symmetric Hamiltonian systems. Let $G$ be the group of symmetries, and $\mathbf{J}:\mathcal M\to \mathfrak{g}^*$ be the momentum map, where $\mathcal M$ is the phase space. 

The first notion is $G$-stability of a relative equilibrium: this is the usual definition of Lyapunov stability but using $G$-invariant open sets.  Specifically, a relative equilibrium $x_e$ is $G$-stable if for every $G$-invariant neighbourhood $V$ of $x_e$ there exists a $G$-invariant neighbourhood $U$ of $x_e$ such that any trajectory intersecting $U$ lies entirely within $V$. Since the dynamics on $\mathcal M$ projects to dynamics on the orbit space (or shape space) $\mathcal M/G$, and the relative equilibrium projects to an equilibrium point, this is equivalent to Lyapunov stability of this projected equilibrium. 

As the momentum is conserved, one can also study stability within a level-set of the momentum, $\mathbf{J}^{-1}(\mu)$ (for the appropriate value of $\mu$). The system on this level set is invariant under $G_\mu$ (by definition of $G_\mu$), and so the natural notion of stability is Lyapunov stability relative to $G_\mu$ on this level set; this is called \emph{leafwise stability}.

A finer notion of stability was introduced by Patrick \cite{Patrick1}, and this is \emph{stability relative to a subgroup} $G'$ of $G$, and in particular the subgroup $G_\mu$ (but here the stability is relative to all perturbations of the initial condition, not just those with the same momentum value). The definition is the same as that above, with $G$ replaced by $G'$.  It is straightforward to show that if $x_e$ is $G'$-stable, then it is also $G$-stable.

Before progressing further, we describe the standard criterion for stability, known as formal stability, and based on Dirichlet's criterion for stability of an equilibrium. Given a symmetric Hamiltonian system, a \emph{symplectic normal space} at a point $x$ is any choice of complement $N_1$ to $T_x(G_\mu\cdot x)$ in $(d\mathbf{J}_x)^{-1}(0)$.  It is a symplectic space, and for free actions it can be naturally identified with the tangent space to the symplectic reduced space. If $x$ is a relative equilibrium, it is said to be \emph{formally stable} if the restriction of the Hessian of the augmented Hamiltonian to (any) symplectic normal space is positive or negative definite. (If $N_1=0$ this condition is empty.)  The question that in general needs addressing is whether formal stability implies stability, and if so whether it is $G_\mu$-stability or simply $G$-stability, or leafwise stability.
 
In order to answer this question, we need to introduce the notions of \emph{split} and \emph{regular} points in momentum space, see \cite{{Lerman},{PRW04}}.

\begin{definition} \label{defregspli} 
(1) Let $\mu \in \mathfrak g^*$ and $G_\mu^0$ the identity component of $G_\mu$.  One says that $\mu$ is \emph{split} if there exists a $G_\mu^0$- invariant complement $\mathfrak{n}_\mu$ to $\mathfrak{g}_\mu$ in $\mathfrak{g}$. \\
(2)  A point $\mu \in \mathfrak{g}^*$ is $\emph{regular}$ if $\dim \mathfrak{g}_\nu = \dim \mathfrak{g}_\mu$ for every $\nu$ in a neighbourhood of $\mu$. \\
(3) A point $\mu\in\mathfrak{g}^*$ is \emph{of Hausdorff type} if there is a neighbourhood $U$ of $G\cdot\mu$ in $\mathfrak{g}^*/G$ within which $G\cdot\mu$ can be separated by disjoint open sets from any other point in $U$.
\end{definition} 

Proposition 5 of \cite{PRW04} states that every regular $\mu$ is of Hausdorff type, and that if $\mu$ is split then it is of Hausdorff type if there exists a $G_\mu^0$-invariant inner product on $\mathfrak{g}_\mu$.  

Patrick \cite{Patrick1} proves that formal stability implies $G_\mu$-stability for relative equilibria for compact groups. This was extended by Lerman and Singer \cite{Lerman}, who showed that formal stability implies $G_\mu$-stability for proper group actions and relative equilibria with split momentum value, provided in addition there is a $G_\mu$-invariant inner product on $\mathfrak g^*$. The more general result due to Patrick, Roberts and Wulff is the following.

\begin{thm}[\cite{PRW04}] Let $x_e$ be a relative equilibrium for a Hamiltonian system with a free and proper action of a symmetry group $G$. 
\begin{enumerate}
\item Suppose $x_e$ is formally stable and that $\mu= J \left(x_e\right)$ is of Hausdorff type. Then $x_e$ is $G$-stable. 
\item  Suppose $x_e$ is $G$-stable with $\mu= J \left(x_e\right)$. If there exists a $G_\mu^0$ invariant inner product on $\mathfrak g^*$, then $x_e$ is $G_\mu^0$-stable.
\end{enumerate}
\label{thm:PRW}
\end{thm}

In our setting with a free action of $G=\SL(2,\mathbb{R})$, for any $\mu\in \mathfrak{sl}\left(2, \mathbb R\right)^*$ the isotropy subgroups $G_\mu=G_\mu^0$ are given in Theorem\,\ref{coadorb}, and one sees $\mu$ is split if and only if it is not parabolic, and every non-zero $\mu$ is both regular and of Hausdorff type. Furthermore, there is a $G_\mu$ invariant inner product on $\mathfrak g^*$ if and only if $\mu$ is elliptic.

Since in our setting the action is free, we restrict attention to that case with no further mention.  (For non-free actions see \cite{MoOR} and references therein.)

\subsection{Two point vortices}

\begin{thm}
 Every trajectory of the two point vortex system in the hyperbolic plane is an $\SL(2, \mathbb R)$-stable relative equilibrium.
\label{G2sta}
\end{thm}

\begin{proof}While this does follow from Theorem \ref{thm:PRW}(1) above (since $\mu\neq0$ and hence regular, and formal stability is trivial as $N_1=0$), this is in fact trivial and needs no general theory: since every point in the phase space is a relative equilibrium, it follows that every point in the orbit space is an equilibrium, and hence in the orbit space every point is stable. 
\end{proof}

To deduce $G_\mu$-stability from Theorem\,\ref{thm:PRW} above, the momentum value must be elliptic.

\begin{thm}
\label{thm:stability of 2 point vortices}
Let $X_e=\left(X_1, X_2\right)\in \mathcal M$ be two point vortices with vorticities $\Gamma_1, \Gamma_2$ and hyperbolic distance $c=d\left({X_1}, {X_2}\right)$. Suppose the momentum value is given by $\mu=\mathbf{J}\left(X_1, X_2\right)\in \mathfrak{sl}\left(2, \mathbb R\right)^*$. Then $X_e$ is $\SL(2, \mathbb R)_\mu$- stable if the vortex strengths are of the same sign or, if they are of opposite sign, the vortices satisfy $c< |\ln|\Gamma_1|-\ln|\Gamma_2|\,|$.
Otherwise it is only leafwise stable.  
\label{corGmu2}
\end{thm}

\begin{proof}
The  momentum value $\mu$ of $X_e$ is elliptic for $\frac{\Gamma_1}{\Gamma_2}>-e^{-c}$ or $\frac{\Gamma_1}{\Gamma_2}<-e^{c}$.
\end{proof}

 \subsection{Three point vortices} \label{stathree}

In this setting, the symplectic normal space $N_1$ is of dimension 2 if $\mu\neq0$ and is zero if $\mu=0$ (because $\dim\mathfrak{g}_\mu=1$ or 3 respectively). 
Pekarsky and Marsden \cite{Pekarsky} find a symplectic normal space for three point vortices on the sphere which we adapt below for the hyperbolic case.  We treat the case of $\mu=0$ at the end of this section. 

\begin{prop} 
\label{ssNpoints}
Let $X_e=(X_1, \dots , X_3) \in \mathcal M$ be a relative equilibrium of a set of point vortices with vorticities $\Gamma=\left(\Gamma_1, \dots , \Gamma_3\right)$ with $\mu\neq0$. Suppose that $D_1$ and $D_2$ are two independent vectors in $\mathbb{R}^3$, such that the plane they span does not contain any of the vortices. Then $N_1=\left<\eta,\zeta \right>$ is a symplectic normal space at $X_e$, where
\begin{equation*} 
\left.\begin{aligned}[c]
\eta&=&\left(a_1D_1\times_{\mathcal{H}} {X_1}, \dots , a_3D_1\times_{\mathcal{H}} {X_3}\right)&,& \\
\zeta&=&\left(b_1D_2\times_{\mathcal{H}} {X_1}, \dots , b_3 D_2\times_{\mathcal H} {X_3}\right),&&
\end{aligned} \right\}
 \qquad \in T_{X_e}\mathcal M
 \end{equation*}
with $a_i,\,b_i$ defined by
\begin{eqnarray*}
\sum_i \Gamma_i a_i D_1 \times_{\mathcal{H}} {X_i}&=&0,\\
\sum_i \Gamma_i b_i D_2 \times_{\mathcal{H}} {X_i}&=&0.
\end{eqnarray*}
\label{sst}
\end{prop}

\myparagraph{Equilateral triangles} The $G$ and $G_\mu$-stability of the equilateral triangle depends on the vorticities only and not on the size of the triangle, and the dependence is the same as for point vortices on the sphere and on the plane.
 
\begin{thm} 
An equilateral configuration $X_e=\left(X_1, X_2, X_3\right)\in \mathcal M$ with vorticities $\Gamma_1, \Gamma_2, \Gamma_3$ and momentum value $\mu=J\left(X_e\right)$ is $\SL(2, \mathbb R)_\mu$-stable if
\begin{equation}
\sum\limits_{i \neq j}\Gamma_i\Gamma_j>0.
\label{eqstable}
\end{equation}
However, if
\begin{equation}
\sum\limits_{i \neq j}\Gamma_i\Gamma_j <0,
\label{equnstable}
\end{equation}
then $X_e$ is $\SL(2,\mathbb R)$-unstable.
\end{thm}

\begin{proof}
From Proposition \ref{sst} 
\begin{equation*} 
\begin{aligned}[c]
\eta:=\left(\begin{array}{c}
\frac{1}{\Gamma_1}\left(D_1 \times_{\mathcal{H}}{X_1}\right)\\ \frac{1}{\Gamma_2}\left(D_1 \times_{\mathcal{H}}{X_2}\right) \\ \left(0,0,0\right)\end{array}\right)
\end{aligned}
\qquad \text{and} \qquad
\begin{aligned}[c]
\zeta:=\left(\begin{array}{c} \left(0,0,0\right)\\ \frac{1}{\Gamma_2}\left(D_2\times_{\mathcal{H}}{X_2}\right)\\ \frac{1}{\Gamma_3} \left(D_2\times_{\mathcal{H}}{X_2}\right)\end{array}\right),
\end{aligned}
\end{equation*}
 with 
$D_1={X_1}+{X_2}$ and $D_2={X_2}+{X_3}$, generate a symplectic normal space $N_1$ at $X_e$. The calculations are very similar to those on the sphere, and one finds that the relative equilibrium is formally stable if and only if $\sum\Gamma_i\Gamma_j>0$.  We know that any equilateral configuration has non-zero and, therefore, regular momentum value $\mu$. Thus the $\SL(2, \mathbb R)$-stability condition follows from Theorem \ref{thm:PRW}(1), provided (\ref{eqstable}) holds.

Moreover, $\mu$ is either elliptic, parabolic or hyperbolic, with determinant 
\begin{equation}
\det \mu=2k\sum\limits_{i \neq j}\Gamma_i\Gamma_j +\sum\limits_{i}\Gamma_i^2,
\label{deteq}
\end{equation}
where $k=-\left\langle X_i,\,X_j\right\rangle_{\mathcal H}>0$ represents the size of the triangle---see  Eq.\,(\ref{eq:cosh relation}). 
Thus for $\mu$ parabolic or hyperbolic only $\left(\ref{equnstable}\right)$ holds, that is $x_e$ is $\SL(2, \mathbb R)$ unstable. Consequently, every parabolic or hyperbolic $x_e$ is $\SL(2, \mathbb R)_\mu$ unstable. If $\mu$ is elliptic, the $\SL(2, \mathbb R)_\mu$ stability results follow from Theorem \ref{thm:PRW}(2).
\end{proof}

\begin{cor}
Every hyperbolic and parabolic equilateral relative equilibrium is $\SL(2, \mathbb R)$, and \emph{a fortiori} $\SL(2, \mathbb R)_\mu$, unstable.\hfill \ensuremath{\square}
\end{cor}

\begin{remark}
The expressions for $G$-stability conditions for three equilateral vortices on the hyperboloid coincide with those for the system on the plane \cite{Aref} and on the sphere \cite{Montaldi1, Pekarsky}.
\end{remark}

Aref \cite{Aref} showed that a three point vortex configuration on the plane, a relative equilibrium with $\sum\limits_{i \neq j}\Gamma_i\Gamma_j =0$ is marginally stable. Meanwhile, 
for the system on the sphere, Marsden, Pekarsky and Shkoller \cite{Pekarsky2} performed numerical integrations and observed changes of the stability for $\sum\limits_{i \neq j}\Gamma_i\Gamma_j =0$. The conjecture that a Hamiltonian bifurcation occurs has also been mentioned in the references \cite{{Montaldi1}, {Pekarsky}}. 

\begin{figure}[tb]
  \centering
  \includegraphics{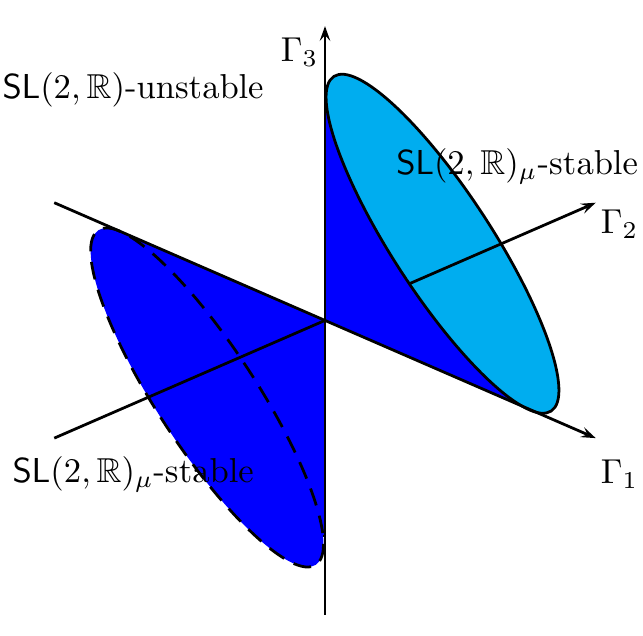}
  \caption{Graph of the cone $\sum \limits_{i\neq j} \Gamma_i \Gamma_j =0$ and corresponding stability properties.}
  \label{eqstagraph}
\end{figure}

%

We have performed numerical integrations in Maple for equilateral configurations on the hyperboloid, which suggest that a bifurcation occurs at $\sum\limits_{i \neq j}\Gamma_i\Gamma_j =0$. Note that this is actually the equation of a cone as shown in Figure \ref{eqstagraph}. The configurations for which $\sum\limits_{i \neq j}\Gamma_i\Gamma_j <0$ have a set of $\Gamma$'s outside this cone, so these points are $\SL(2, \mathbb R)$-unstable. On the other hand, any equilateral configuration with $\Gamma$'s inside the cone is a $\SL(2, \mathbb R)_\mu$-stable relative equilibrium, and follows the trajectory of an ellipse rotating around its momentum value and is therefore periodic. 

\myparagraph{Isosceles geodesic relative equilibria}

\begin{thm} 
Let $X_e=\left(X_1, X_2, X_3\right)\in \mathcal M$ be a configuration of point vortices lying on a geodesic of the hyperboloid with vorticities $\Gamma_1, \Gamma_2, \Gamma_3=\Gamma_1$ and momentum  $\mu=J\left(X_e\right)$.  Suppose  $\langle {X_1},  {X_2}\rangle_{\mathcal H}=\langle {X_2},  {X_3}\rangle_{\mathcal H}=a\neq \frac{\Gamma_2}{2\Gamma_1}$ then $\mu$ is elliptic. Furthermore, let
\begin{equation}
A\left(\Gamma_1, \Gamma_2, a\right)=\frac{1}{\Gamma_1}\Bigl(512 A_1\left(\Gamma_1, \Gamma_2, a\right)+A_2\left(\Gamma_1, \Gamma_2, a\right)\Bigr),
\label{condis}
\end{equation}
with
\begin{eqnarray*}
A_1\left(\Gamma_1, \Gamma_2, a\right)&=&\Gamma _1^2a^9-2\Gamma _1a^8\Bigl(\Gamma _1-\frac{\Gamma _2}{4} \Bigr) +a^7 \Bigl( -\frac{5}{4}\Gamma _1^2+2\Gamma _1\Gamma _2 \Bigr)  \nonumber \\
&& +2 a^6\Bigl(\Gamma _1+\frac{\Gamma _2}{4} \Bigr)  \Bigl( \Gamma _1-\Gamma _2 \Bigr)+\frac{\Gamma _1a^5}{4} \left( \Gamma _1-8\Gamma _2 \right) \nonumber\\
&& +\frac{\Gamma _2 a^4}{16}\left( 8\Gamma _2+\Gamma _1 \right) -\frac{\Gamma _1\Gamma _2}{32}\left(a^2-\frac{1}{2}\right),
\end{eqnarray*}
and
\[A_2\left(\Gamma_1,\Gamma_2, a\right)=\Gamma _1a^5+\frac{1}{2}\Gamma _2a^{4}  -\frac{5}{4}\Gamma _1{a}^{3}-\frac{11}{8}\Gamma_2{a}^{2}+\frac{1}{4}\Gamma _1a-\frac{1}{8}\Gamma_{{2}}.\]
If $A\left(\Gamma_1, \Gamma_2, a\right)>0$ then $X_e$ is $\SL(2, \mathbb R)_\mu$-stable. Conversely if $A\left(\Gamma_1, \Gamma_2, a\right)<0$ then $X_e$ is $\SL(2, \mathbb R)$-unstable.
In addition, if $X_e$ is also an equilibrium point, then $X_e$ is $\SL(2, \mathbb R)_\mu$-stable for all $a \notin (-1.191, -1.106)$. 
Finally, if $a= \frac{\Gamma_2}{2\Gamma_1}$ then $\mu=0$ and hence not a regular point.
\label{staisos}
\end{thm}

\begin{proof}
In the proof of Theorem \ref{georelt} we showed that $\det \mu>0$ for all $a\neq\frac{\Gamma_2}{2\Gamma_1}$, otherwise $\mu=0$. By Theorem \ref{sst}, a symplectic normal space $N_1$ to $X_e$ is generated by
\begin{equation*} 
\begin{aligned}[c]
\eta:=\left(\begin{array}{c}
\frac{1}{\Gamma_1}\left(D_1\times_{\mathcal{H}} {X_1} \right)\\ 
\frac{1}{\Gamma_2}\left(D_1\times_{\mathcal{H}} {X_2} \right)\\ 
\left(0,0,0\right)%
\end{array}\right)
\end{aligned}
\qquad \text{and} \qquad
\begin{aligned}[c]
\zeta:=\left(\begin{array}{c}
\frac{1}{\Gamma_1}\left(D_2\times_{\mathcal{H}} {X_1} \right)\\ 
\frac{2k}{\Gamma_2}\left(D_2\times_{\mathcal{H}} {X_2} \right)\\
\frac{1}{\Gamma_3}\left(D_2\times_{\mathcal{H}} {X_3} \right)\\
\end{array}\right),
\end{aligned}
\end{equation*}
where $D_1={X_1}+{X_2}$, $D_2=\left(0,1,0\right)$ and $k=-\sqrt{x_1^2+1}$. Given that $\mu$ is regular and elliptic, the stability condition $\left(\ref{condis}\right)$ is derived from testing the definiteness of the Hessian of $\left(\ref{extended}\right)$ restricted to $N_1$. 

Recall from Section \ref{relthreesec}, that an isosceles geodesic relative equilibrium is also in equilibrium provided $\Gamma_2=\frac{\Gamma_1a}{1-a}$. Given that $a<-1$ for all $X_e$, we obtain $A\left(\Gamma_1, \Gamma_2, a\right)>0$ for $a \notin \left(-1.191, -1.106\right)$.
\end{proof}

\myparagraph{Zero momentum}
The symplectic normal space of $\mu=0$ is trivial. As discussed in Example $4$ of \cite{PRW04}, a configuration with $\mu=0$ is trivially leafwise stable and, $G$-stable if the angular velocity $\xi$ points into the null-cone $\mathfrak C$. Simple calculations show that the angular velocity of a geodesic configuration with momentum value $\mu=0$ is always elliptic, hence $\mu=0$ is $\SL(2, \mathbb R)$-stable.

Despite the complexity of $\left(\ref{condis}\right)$, we can get an idea of the stability of $X_e$ by looking at Figure \ref{a-g2}, where the stability regions for $\Gamma_1=\Gamma_3=1$ are plotted.  The second diagram is a close up of the bottom-right portion of the first diagram and the stability conditions can be seen with more detail for vortices that are close to each other, that is for small values of $a$. The  dashed blue line represents $a=\frac{\Gamma_2}{2\Gamma_1}$, in which case $\mu=0$ and $X_e$ is $SL\left(2,\mathbb R\right)$-stable as the angular velocity $\xi$ is elliptic. 

\begin{figure}[tb]
  \centering
          \includegraphics[width=\textwidth]{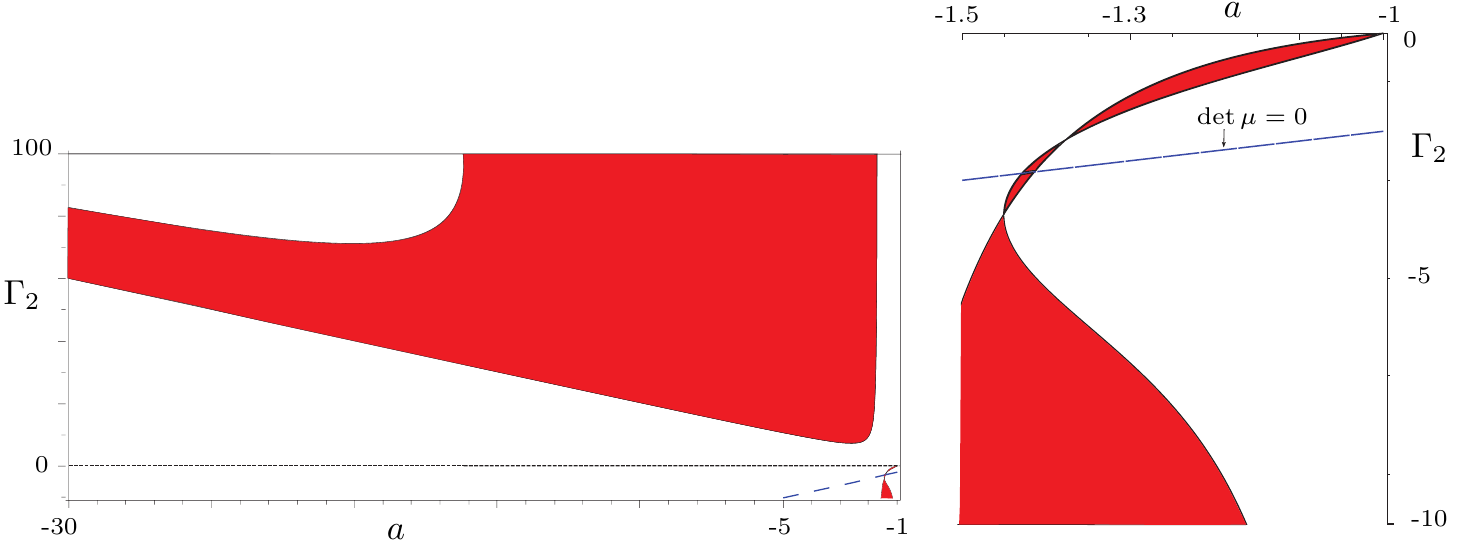}
    \caption{Stability of isosceles geodesic relative equilibria for $\Gamma_1=\Gamma_3=1$ as a function of $\Gamma_2$ and $a$. The white region represents $\SL(2, \mathbb R)_\mu$-stable relative equilibria, and the red region represents $\SL(2, \mathbb R)$-unstable relative equilibria}
    \label{a-g2}
  \end{figure}
  
  \myparagraph{Geodesic configuration with three different lengths}
Additional information of this general case can be found in \cite{thesis}. For this type of configuration the computation of the Hessian is rather involved, and further analysis is required to provide general stability criteria for relative equilibria of this system.  Nevertheless it is of particular use the fact that as for the isosceles case, a relative equilibrium must have either elliptic or zero momentum value. 

\myparagraph{Acknowledgements} We would like to thank the referee for comments which helped improve the presentation. 
C~N-G was supported by CONACyT grant 307762.

\let\oldbibliography\thebibliography
  \renewcommand{\thebibliography}[1]{%
  \oldbibliography{#1}
  \setlength{\itemsep}{0pt}
  \small
  } 

\end{document}